\documentclass[a4paper,11pt]{article}       

\newcounter{CounterTodo}
\usepackage{xcolor}

\usepackage{graphicx}
\usepackage{booktabs}
\usepackage{hyperref}
\usepackage{enumerate}
\usepackage{algorithm}
\usepackage{algpseudocode}
\usepackage[normalem]{ulem}
\usepackage{authblk}
\usepackage{amsmath}
\usepackage{amsthm}
\usepackage{fullpage}
\usepackage{lineno}
\usepackage{setspace}
\usepackage{lineno}

\newtheorem{theorem}{Theorem}
\newtheorem{lemma}[theorem]{Lemma}
\newtheorem{corollary}[theorem]{Corollary}
\newtheorem{proposition}[theorem]{Proposition}
\newtheorem{observation}[theorem]{Remark}
\newtheorem{definition}{Definition}

\newcommand{\noc}{noc}
\newcommand{\noi}{noi}
\newcommand{\Noc}{\scriptsize\scriptsize\textsc{NOC}}

\title{Counting $P_3$-convex sets in graphs}

\author[1]{Mitre C.\ Dourado}
\author[2]{Luciano N.\ Grippo}
\author[3,4]{Min Chih Lin}
\author[5]{F{\'a}bio Protti}

\affil[1]{Instituto de Computa\c{c}\~ao, Universidade Federal do Rio de Janeiro, Brazil}
\affil[2]{Universidad Nacional de General Sarmiento. Instituto de Ciencias; Argentina. CONICET, ICI-UNGS, Buenos Aires, Argentina}
\affil[3]{CONICET, Instituto de C\'alculo, Buenos Aires, Argentina}
\affil[4]{Departamento de Computaci\'on, Universidad de Buenos Aires, Buenos Aires, Argentina}
\affil[5]{Instituto de Computa\c{c}\~ao, Universidade Federal Fluminense, Niter\'oi, Brazil}
\begin{document}

\onehalfspace

\maketitle
\begin{abstract}
	\textcolor{black}{
	We study the $P_3$-convexity, the path convexity generated by all three-vertex paths, and focus on the problem of counting the $P_3$-convex vertex sets of a graph $G$, denoted by $\noc(G)$. First, we settle the associated extremal question: we characterize the $n$-vertex graphs maximizing $\noc(G)$ among all graphs and determine the connected extremal graphs. Next, we investigate computational complexity and show that counting $P_3$-convex sets is $\#\mathsf{P}$-complete already on split graphs, even under additional structural restrictions. On the positive side, we identify two tractable subclasses, namely trees and threshold graphs, and obtain linear-time algorithms for both. Finally, we design nontrivial exact exponential-time algorithms for general graphs, combining structural decomposition, propagation rules capturing forced consequences of $P_3$-convexity, and fast counting of independent sets in auxiliary graphs. The resulting strategy becomes particularly effective on graph classes where large independent sets are guaranteed and can be found efficiently.
}
\end{abstract}

\medskip
\noindent\textbf{Keywords:}
$P_3$-convexity; counting $P_3$-convex sets; extremal graphs; $\#\mathsf{P}$-completeness; exact exponential-time algorithms.

\section{Introduction} \label{sec:int}
\subsection{State of the art}
\textcolor{black}{ Graph convexity has attracted increasing attention from both the mathematics and computer science communities, driven by a variety of challenging structural and algorithmic questions. A \emph{convexity} on a finite graph $G$ is a pair $(V(G), \mathcal{C})$, where $\mathcal{C}$ is a family of subsets of $V(G)$ satisfying the following conditions: $\emptyset \in \mathcal{C}$, $V(G) \in \mathcal{C}$, and $\mathcal{C}$ is closed under intersections; that is, $V_1 \cap V_2 \in \mathcal{C}$ for every $V_1, V_2 \in \mathcal{C}$. Each set in the family $\mathcal{C}$ is called a \emph{$\mathcal{C}$-convex set}. In 1986, Farber and Jamison published a seminal article on convexities in graphs from a theoretical and structural perspective~\cite{FarberJ1986}. The intense activity in this line of research created the need to compile a compendium of the most important published results. In 1993, van de Vel published a book on this topic~\cite{vandevel93}, and in 2013 Pelayo did so as well~\cite{Pelayo2013}. More recently, an algorithmic treatment of the topic has appeared in a book by Ara{\'u}jo, Dourado, Protti, and Sampaio~\cite{AraujoDouradoProttiSampaio2025IntroGraphConvexity}.}

\textcolor{black}{Many of the most studied graph convexities arise from suitable choices of families of paths in the graph. Let $\mathcal{P}$ be a set of paths in $G$. Let $\mathcal{C}$ be the family of all vertex subsets $S \subseteq V(G)$ such that, for every path $P \in \mathcal{P}$ with both endpoints in $S$, all vertices of $P$ also belong to $S$. It is easy to verify that $(V(G), \mathcal{C})$ defines a convexity on $G$, and $\mathcal{C}$ is called the \emph{path convexity generated by $\mathcal{P}$}. In the literature, several special families of paths $\mathcal{P}$ have been considered, such as shortest paths~\cite{Douchet1988, KanteMMS2019, AnandChChDHN2020} (geodetic convexity), triangle paths~\cite{DouradoS2016, ChangatM1999}, induced pahts (monophonic convexity)~\cite{DouradoPS2010, CostaDS2015, BenavidesCDRS2016}, all paths~\cite{ChangatKM2001},  paths on three vertices~\cite{BarbosaCDRS2012, CentenoPRP2013}, and induced paths on three vertices~\cite{AraujoSdosSS2018, DouradoPPR2022}.}

\textcolor{black}{The \emph{$P_3$-convexity} is the path convexity generated by all paths of length two, that is, by all three-vertex paths. A subset $S \subseteq V(G)$ is \emph{$P_3$-convex} if every vertex $v \in V(G)\setminus S$ has at most one neighbor in $S$. A precursor to this convexity can be found in the popular mathematics book by Bollob\'as (2006)~\cite{bollobas2006art}, where the author considers the $P_3$-convex hull on a grid in the context of studying the spread of a disease on a grid. In this paper, we focus on the study of this convexity.}

\textcolor{black}{The general problem of counting a given type of structure in a graphs is widespread in the research community. In 1979, Valiant published a pioneering article on the complexity theory of enumeration problems~\cite{Valiant1979}. Clearly, effectively listing all graphs (or sets) possessing a specific property $\pi$ yields, as a direct consequence, the exact count of such structures. Formally, an enumeration algorithm $\mathcal{A}$ outputs a sequence $G_1, G_2, \dots, G_k$ of all graphs satisfying $\pi$, such that $G_i \not\equiv G_j$ for $1 \leq i < j \leq k$. Various criteria exist to evaluate the efficiency of these algorithms \cite{Kimelfeld2005}. The least restrictive is \textit{polynomial total time}, where the execution time is bounded by a polynomial function of the combined size of all outputs $G_1, \dots, G_k$. A more stringent requirement is \textit{polynomial delay}, which dictates that the time elapsed between the generation of any two consecutive elements is polynomial only in the size of the subsequent output. Notable examples of algorithms achieving polynomial delay include those for generating all minimum spanning trees of a weighted graph \cite{Eppstein1995}, all maximal independent sets \cite{Lawler1980}, all elementary cycles \cite{Lauer1976}, and all cographs with $n$ vertices  \cite{Jones2018}. Furthermore, Courcelle \cite{Courcelle2009} established sufficient conditions to guarantee \textit{linear delay} when enumerating certain structures within graphs of bounded tree-width. A crucial issue in calculating the complexity of enumeration algorithms is the computational effort spent on avoiding the listing of duplicate elements. The algorithm described by \cite{Lawler1980} is a good example of how to handle duplicate discard while maintaining polynomial delay. }
	
\textcolor{black}{To the best of our knowledge, the problem of counting or enumerating convex sets has received little attention. A precursor to this kind of problem in the context of graph convexities can be found in~\cite{HaglinW1996}, where the authors study the convexity of directed three paths in tournaments by presenting an algorithm that finds all the convex sets in $O(n^4)$ time. For $\#\mathsf{P}$-complete problems, a reasonable alternative is to design exact exponential-time algorithm. Examples include $\#\textsc{Stable Sets}$~\cite{GaspersLee2023}, $\#2-\textsc{Sat}$ and $\#3-\textsc{Sat}$~\cite{DahllofVW2005, Kutzov2007}, and $\#\textsc{Perfect Matchings}$~\cite{BjorklundH2008}.}

\textcolor{black}{This article is devoted to the problem of counting the number of $P_3$-convex sets in a graph. We address the following two questions:
\begin{enumerate}
	\item Given a graph on $n$ vertices (connected or not), which graphs maximize the number of $P_3$-convex sets? (see Corollary~\ref{cor: general graphs} and Theorem~\ref{thm: optimum graph}).
	\item What is the computational complexity of counting the $P_3$-convex sets of a graph? (see Theorem~\ref{thm: hardness result} and Corollary~\ref{cor:oneinducedK14}).
\end{enumerate}}
\textcolor{black}{Since the answer to the second question turns out to be $\#\mathsf{P}$-complete, two further problems arise:
\begin{enumerate}
	\item Identify a graph class for which this problem can be solved in polynomial time (see Theorems~\ref{thm: counting P_3 convex sets in trees} and~~\ref{thm: thershold graphs}).
	\item Design a non-trivial exact exponential-time algorithm to count the $P_3$-convex sets of a graph (see Section~\ref{subsec:exact-algorithms}).
\end{enumerate}
In this article, we answer both questions stated above and propose solutions to the two problems that arise from the $\#\mathsf{P}$-completeness of the second one.}

\textcolor{black}{The article is organized as follows. In Section~\ref{sec: extremal problem} we address the extremal question and study which $n$-vertex graphs maximize the number of $P_3$-convex sets. We first develop a dynamic-programming approach for trees, then compare stars and paths, and finally characterize the connected extremal graphs. In Section~\ref{sec: Complexity results} we turn to computational complexity: we prove that counting $P_3$-convex sets is $\#\mathsf{P}$-complete (already on split graphs) and present a linear-time solvable case on threshold graphs. Finally, in Section~\ref{subsec:exact-algorithms} we develop exact exponential-time algorithms for general graphs, improving over the naive $O^*(2^n)$ algorithm via structural decomposition, propagation rules, and fast independent-set counting, and we discuss refinements for graph classes with guaranteed large independent sets.
}
%
%
\subsection{Basic definitions and notations}
\textcolor{black}{We consider finite, simple, and undirected graphs. For a graph $G$, we denote its vertex set by $V(G)$ and its edge set by $E(G)$. For a vertex $v\in V(G)$, the \emph{degree} of $v$ in $G$ is $d_G(v)=|N_G(v)|$ (or simply $d(v)$ when $G$ is clear from the context). The \emph{minimum degree} and \emph{maximum degree} of $G$ are
\[
\delta(G)=\min\{d_G(v): v\in V(G)\}
\qquad\text{and}\qquad
\Delta(G)=\max\{d_G(v): v\in V(G)\},
\]
respectively.}

\textcolor{black}{A \emph{subgraph} of $G$ is any graph $H$ such that $V(H)\subseteq V(G)$ and $E(H)\subseteq E(G)$. For $U\subseteq V(G)$, the \emph{subgraph of $G$ induced by $U$} is denoted by $G[U]$; it has vertex set $U$ and edge set $\{xy\in E(G): x,y\in U\}$. Any graph $H$ with $V(H)\subseteq V(G)$ and $E(H)\subseteq E(G)$ is a \emph{subgraph} of $G$, and it is \emph{induced} if $H=G[V(H)]$. A \emph{spanning subgraph} of $G$ is a subgraph $H$ with $V(H)=V(G)$. A \emph{tree} is a connected acyclic graph, and a \emph{spanning tree} of $G$ is a spanning subgraph of $G$ that is a tree.}

\textcolor{black}{For a vertex $v\in V(G)$, the \emph{(open) neighborhood} of $v$ in $G$ is $N_G(v)=\{u\in V(G): uv\in E(G)\}$, and its \emph{closed neighborhood} is $N_G[v]=N_G(v)\cup\{v\}$. A set $I\subseteq V(G)$ is an \emph{independent set} if no two vertices of $I$ are adjacent. A \emph{clique} in a graph $G$ is a set of vertices $C\subseteq V(G)$ such that every two distinct vertices in $C$ are adjacent.}

\textcolor{black}{For graph operations, $G-v$ denotes the graph obtained from $G$ by deleting the vertex $v$ (and all edges incident to $v$). More generally, for $S\subseteq V(G)$, we write $G-S$ for the graph obtained by deleting all vertices of $S$. Similarly, for an edge $e\in E(G)$, $G-e$ denotes the graph obtained by deleting $e$, and for $F\subseteq E(G)$, we write $G-F$ for the graph obtained by deleting all edges in $F$. The \emph{disjoint union} of graphs $G$ and $H$ is denoted by $G+H$. Let $s\cdot G$ denote the disjoint union of $s$ copies of $G$.}

\textcolor{black}{We use standard notation for some families of graphs: $K_n$ is the complete graph on $n$ vertices, $C_n$ is the cycle on $n$ vertices, and $P_n$ is the path on $n$ vertices. The complete bipartite graph with parts of size $a$ and $b$ is denoted by $K_{a,b}$; in particular, the \emph{star} on $n$ vertices is $K_{1,n-1}$.}

\textcolor{black}{A graph $G$ is a \emph{split graph} if its vertex set can be partitioned as
$V(G)=K\cup S$, where $K$ induces a clique and $S$ induces an independent set. Any such partition $(K,S)$ is called a \emph{split partition} of $G$.}

\section{Extremal problem}\label{sec: extremal problem}
\textcolor{black}{The graphs attaining the maximum number of convex sets are exactly the $P_3$-free graphs, in which every subset of vertices is convex.
	In contrast, the existence of an induced $P_3$ with vertices $u - v - w$ implies that $\{u,w\}$ fails to be convex, since convexity forces the inclusion of $v$.
	For connected graphs, the only $P_3$-free graphs are the complete graphs, which are therefore the unique extremal connected graphs.}

We begin by recalling the notion of $P_3$-convexity and introducing the notation
$\noc(G)$.

\begin{definition}[$P_3$-convexity and the function $\noc(G)$]
Let $G=(V,E)$ be a graph.  
A set $S \subseteq V$ is \emph{$P_3$-convex} if for every path $x$--$z$--$y$ in $G$
with $x,y \in S$, we also have $z\in S$.  
We denote by $\noc(G)$ the number of $P_3$-convex sets of $G$, that is,
\[
\noc(G)
 \;=\;
\bigl|\{\, S \subseteq V : S \text{ is $P_3$-convex in } G \,\}\bigr|.
\]
\end{definition}
We also use the basic color interpretation:
vertices in a $P_3$-convex set are considered \emph{black},  
and vertices outside the set are considered \emph{white}.  
This two-color viewpoint will later be refined for rooted trees.

\textcolor{black}{We begin by studying the connected graphs on $n$ vertices that maximize the number of $P_3$-convex sets. First, we observe the crucial property that this number is nondecreasing under edge deletions.}
\begin{lemma} \label{lem:deleteEdge}
Let $G$ be a graph and let $uv\in E(G)$ be an edge. 
Then
\[
\noc(G - uv) \;\ge\; \noc(G).
\]
\end{lemma}
%
%
\textcolor{black}{
Applying Lemma~\ref{lem:deleteEdge} iteratively, one sees that edgeless graphs attain the maximum possible number of $P_3$-convex sets. However, they are not the only extremal graphs: any graph $G$ with maximum degree $\Delta(G)\le 1$ is also extremal. Indeed, such graphs are $P_3$-free and each of their connected components has at most two vertices; consequently, every subset of vertices is $P_3$-convex. In contrast, any graph with maximum degree at least $2$ contains either an induced $P_3$ or a triangle. In both cases there exist vertices $u$ and $w$ with a common neighbor $v$, implying that the set $\{u,w\}$ is not $P_3$-convex. Thus graphs with $\Delta(G)\le 1$ are exactly the extremal ones. Hence, the following result holds
\begin{corollary}\label{cor: general graphs}
	Let $G$ be a graph on $n$ vertices. Then $\noc(G)\le 2^n$. Moreover, equality holds if and only if $G$ is isomorphic to $s\cdot K_1 + t\cdot K_2$ with $2t + s = n$.
\end{corollary}}
\textcolor{black}{We now turn our attention to connected extremal graphs, and conclude by presenting the main result of this section in Subsection~\ref{subsec:extremal}.
The concepts introduced and the results established in the preceding subsections together lead to their characterization.}
\subsection{Trees} \label{subec:trees}
For trees, \textcolor{black}{enumerating $P_3$-convex sets} admits a \textcolor{black}{simple recursive} formulation. Given a tree $T$, we choose an arbitrary root and orient all edges away from it. Each vertex defines the \textcolor{black}{unique subtree induced by its descendants}. \textcolor{black}{We will refine the white color into two states, depending on whether a white vertex has a black parent, as follows.}
\[
\mathsf{B}\ (\text{black}), \qquad
\mathsf{G}\ (\text{white with a black parent}), \qquad
\mathsf{W}\ (\text{white with no black parent}).
\]
\textcolor{black}{Given a directed rooted tree $T$, let $\mathrm{root}(T)$ denote its root.}
\begin{definition}[Auxiliary function $\noc(T,\alpha)$]
	Let $T$ be a rooted \textcolor{black}{directed} tree.  
	For $\alpha \in \{\mathsf{B},\mathsf{G},\mathsf{W}\}$,  \textcolor{black}{let
	$\noc(T,\alpha)$ denote} the number of $P_3$-convex sets of $T$ in which $\mathrm{root}(T) \text{ has color }\alpha$.
	\end{definition}
	\textcolor{black}{Since $T$ is a tree, every $P_3$ is necessarily induced. The value $\noc(T,\alpha)$ can be computed recursively from the subtrees rooted at the children of $\mathrm{root}(T)$.}
\begin{algorithm}[H]
\caption{Computing $\noc(T,\alpha)$}\label{algo: P_3 convex sets in trees}
\begin{algorithmic}[1]
\Function{noc}{$T$, $\alpha$}
    \If{$T$ is trivial}
        \State \Return $1$
    \EndIf

    \State $num \gets 1$
    \State $numw \gets 1$
    \State Let $T_1,\dots,T_k$ be the subtrees of $T - \mathrm{root}(T)$

    \For{$i = 1$ to $k$}
        \If{$\alpha = \mathsf{B}$}
            \State $num \gets num \cdot (\Noc(T_i,\mathsf{G}) + \Noc(T_i,\mathsf{B}))$

        \ElsIf{$\alpha = \mathsf{G}$}
            \State $num \gets num \cdot \Noc(T_i,\mathsf{W})$

        \ElsIf{$\alpha = \mathsf{W}$}
            \State $nb[i] \gets \Noc(T_i,\mathsf{B})$
            \State $nw[i] \gets \Noc(T_i,\mathsf{W})$
            \State $numw \gets numw \cdot nw[i]$
        \EndIf
    \EndFor

    \If{$\alpha = \mathsf{W}$}
        \State $num \gets numw$
        \For{$i = 1$ to $k$}
            \State $num \gets num + (numw / nw[i]) \cdot nb[i]$
        \EndFor
    \EndIf

    \State \Return $num$
\EndFunction
\end{algorithmic}
\end{algorithm}
\textcolor{black}{We omit the details of the correctness proof of Algorithm~\ref{algo: P_3 convex sets in trees} in computing $\noc(T,\alpha)$}. Since the root cannot be colored $\mathsf{G}$ (\textcolor{black}{that is, a} white vertex with a black parent),
the total number of $P_3$-convex sets of $T$ is
\[
\noc(T) = \noc(T,\mathsf{B}) + \noc(T,\mathsf{W}).
\]
\textcolor{black}{Additionally, Algorithm~\ref{algo: P_3 convex sets in trees} process each subtree $T_i$ once in a bottom–up traversal of $T$. Therefore, we obtain the following complexity result.	
\begin{theorem}\label{thm: counting P_3 convex sets in trees}
	The number of $P_3$-convex sets of a tree $T = (V,E)$ can be computed in $O(|V| + |E|)$ time.   
\end{theorem}}
\textcolor{black}{We will revisit the complexity aspects of enumerating $P_3$-convex sets in Section~\ref{sec: Complexity results}.}
%
%
%
%
%
%
%
%
\subsection{Stars $K_{1,n-1}$ and Paths $P_n$}
\textcolor{black}{The following technical lemma, whose proof is straightforward  and therefore omitted, gives the number of $P_3$-convex sets of the candidate graph on $n$ vertices that maximizes this quantity.}
\begin{lemma}\label{lem:star}
	For $n \ge 1$, it holds that $\noc(K_{1,n-1}) = 2^{n-1} + n$.
\end{lemma}
%
%
\textcolor{black}{The next result rules out unicyclic graphs on $n$ vertices (connected graph with the same number of vertices and edges) as candidates for maximizing the number of $P_3$-convex sets among connected graphs on $n$ vertices with $n\ge 3$.}
\begin{lemma}~\label{lem: unicylic graphs}
	Let $G=(V,E)$ be a connected graph with $|V|=|E|=n\ge 3$.  
	Then for every spanning tree $T$ of $G$,
	\[
	\noc(G) < \noc(T).
	\]
\end{lemma}

\begin{proof}
	Since $G$ is connected and $3\ge |V|=|E|$, it contains exactly one cycle. Let \textcolor{black}{$C$} be the cycle, and let $uv$ be any edge of $C$. Then $T=G-uv$ is a spanning tree of $G$. Indeed, every spanning tree of $G$ can be obtained by deleting exactly one edge of its unique cycle.

	\textcolor{black}{By Lemma~\ref{lem:deleteEdge}, we have $\noc(G)\le \noc(T)$}. To prove strict inequality, consider
	\[
	S = V(C)\setminus\{u\}.
	\]
	\textcolor{black}{On one hand, since $T$ is acyclic, any vertex $z\notin S$ has at most one neighbor in $S$. Hence, no $P_3$ with endpoints in $S$ has its middle vertex outside $S$, and therefore, $S$ is $P_3$-convex in $T$. On the other hand, in $G$ the vertex $u$ has two neighbors in $S$ along the cycle $C$, and therefore any convex set of $G$ containing those two neighbors must also contain $u$. Consequently, $S$ is not $P_3$-convex in $G$.}

	In summary, $T$ has a $P_3$-convex set that $G$ does not, implying $\noc(G)<\noc(T)$.
\end{proof}
\textcolor{black}{We can now derive the following result, which guarantees that any $n$-vertex connected graph maximizing the number of $P_3$-convex sets must be a tree.}
\begin{corollary}\label{cor:cycle}
Let $G=(V,E)$ be any connected graph with $|E|\ge |V|\ge 3$.  
Then every spanning tree $T$ of $G$ satisfies
\[
\noc(G) < \noc(T).
\]
\end{corollary}

\begin{proof}
	\textcolor{black}{If $|E|=|V|$, then the statement follows from Lemma~\ref{lem: unicylic graphs}.}

	\textcolor{black}{Assume now that  $|E| > |V|$.  Let $T=(V,E_T)$ be any spanning tree of $G$. Since $E_T\subset E$ and $|E|>|V|$, the set $E\setminus E_T$ is nonempty. Pick an edge $e\in E\setminus E_T$ and consider the spanning subgraph
	\[
	G'=(V,\, E_T \cup \{e\}).
	\]
	}
	\textcolor{black}{The graph $G'$ is unicyclic. Thus, by combining Lemmas~\ref{lem:deleteEdge} and~\ref{lem: unicylic graphs}, we obtain
	\[
	\noc(G)\le\noc(G') < \noc(T),
	\]
	and therefore the result hods.}
%
%
\end{proof}
We derive a simple recurrence for $Z_n=\noc(P_n)$, where $P_n$ is the path on $n$ vertices. Root $P_n$ at one of its endpoints and \textcolor{black}{set}
\[
B_n = \noc(P_n,\mathsf{B}), \qquad
W_n = \noc(P_n,\mathsf{W}), \qquad
G_n = \noc(P_n,\mathsf{G}),
\]
where $\mathsf{B}$, $\mathsf{W}$, and $\mathsf{G}$ indicate, respectively, that the root is black, white with no black parent, and white with a black parent. Since the root has no parent, it cannot be in state $\mathsf{G}$, and therefore
\[
Z_n = \noc(P_n) = B_n + W_n.
\]
\paragraph{Recurrence relations.}
Because every vertex of a path has at most one child, the general tree
recurrences simplify to
\[
B_n = B_{n-1} + G_{n-1}, \qquad
W_n = W_{n-1} + B_{n-1}, \qquad
G_n = W_{n-1},
\]
with initial values
\[
B_1 = W_1 = G_1 = 1.
\]
Eliminating $B_n$, $W_n$, and $G_n$ yields the linear recurrence
\[
\boxed{
Z_n = 2Z_{n-1} - Z_{n-2} + Z_{n-3}, \qquad n\ge 4,
}
\]
with
\[
Z_1=2,\qquad Z_2=4,\qquad Z_3=7.
\]
We compare $Z_n$ with the number of $P_3$-convex sets of the star $K_{1,n-1}$.
\textcolor{black}{By} Lemma~\ref{lem:star},
\[
\noc(K_{1,n-1}) = 2^{\,n-1} + n.
\]
\textcolor{black}{Therefore, the following proposition holds; its proof is straightforward by induction on $n$ for $n\ge 6$, and the remaining cases are covered in Table~\ref{table: first 10 noc of P_n}}
\begin{proposition}\label{pro:path}
	For every integer $n\ge 1$,
	\[
	\noc(P_n) \;\le\; \noc(K_{1,n-1}) = 2^{\,n-1} + n,
	\]
	and equality holds if and only if $n\le 5$.
\end{proposition}
%
%
%
%
%
%
%
%
%
\begin{table}[H]
\centering
\begin{tabular}{c|cccccccccc}
$n$                 & 1 & 2 & 3 & 4  & 5  & 6  & 7  & 8   & 9   & 10  \\ \hline
$\noc(P_n)$         & 2 & 4 & 7 & 12 & 21 & 37 & 65 & 114 & 200 & 351 \\
$\noc(K_{1,n-1})$   & 2 & 4 & 7 & 12 & 21 & 38 & 71 & 136 & 265 & 522 \\
\end{tabular}
\caption{Comparison of $\noc(P_n)$ with $\noc(K_{1,n-1}) = 2^{\,n-1} + n$.}\label{table: first 10 noc of P_n}
\end{table}
\subsection{Connected extremal graphs with maximum $\noc$} \label{subsec:extremal}
\textcolor{black}{We need only three ingredients, stated below, to be ready to state and prove the main result of this section. The first two facts can be follow directly from the definitions of a tree and a $P_3$-convex set.}
\begin{proposition}\label{prop: technical}
Let $T$ be a tree on $n$ vertices.  
Then $T$ satisfies at least one of the following conditions:
\begin{enumerate}
    \item $T$ is the star $K_{1,n-1}$;
    \item $T$ is the induced path $P_n$;
    \item $T$ has a vertex $v$ with $\deg(v)\ge 3$ and $T$ has no universal vertex
          (i.e., there is no vertex of degree $n-1$).
\end{enumerate}
\end{proposition}
%
%
%
%
%
%
\begin{observation}\label{obs:subtree-convex}
	Let $T$ be a tree and let $U$ be the vertex set of any subtree of $T$.
	Then $U$ is a $P_3$-convex set of $T$. In particular, any vertex set of a subpath of $T$ is a $P_3$-convex set.
\end{observation}
%
%
%
%
\begin{proposition}\label{pro:casi}
%
\textcolor{black}{Let $T$ be a tree on $n$ vertices, and assume that $T$ is not isomorphic 
to a star $K_{1,n-1}$ nor to a path $P_n$.
Let $r$ be any leaf of $T$ and regard $T$ as rooted at $r$.
Let $r'$ be the unique neighbor of $r$, and let $T'=T-r$, regarded as rooted at $r'$.}

Then:
\begin{enumerate}
    \item $\noc(T',\mathsf{W}) - \noc(T',\mathsf{G}) \;\ge\; n-1$.
    \item $\noc(T) \;<\; 2^{\,n-1} + n \;=\; \noc(K_{1,n-1})$.
\end{enumerate}
\end{proposition}

\begin{proof}
%
	\noindent
	\textbf{Part (1).}
	Root $T'$ at $r'$.  
	For each vertex $v\in V(T')\setminus\{r'\}$, define $S(v)$ as the set of all vertices on the unique $v$--$r'$ path, excluding $r'$ but including $v$.

	Each $S(v)$ induces a subtree, hence\textcolor{black}{, by Remark~\ref{obs:subtree-convex},} is \textcolor{black}{a} $P_3$-convex \textcolor{black}{set}. Moreover:
	\begin{itemize}
		\item $S(v)$ contains exactly one neighbor of the root \(r'\), so it is counted by
	$\noc(T',\mathsf{W})$.
		\item $S(v)$ is \emph{not} counted by $\noc(T',\mathsf{G})$, because in the state $\mathsf{G}$ the root must have all children white, i.e., none of its neighbors may be included.
	\end{itemize}
	Since $T'$ has $n-1$ vertices and we exclude $r'$, the sets $S(v)$ account for exactly $n-2$ distinct $P_3$-convex sets counted \textcolor{black}{that are counted in $\noc(T',\mathsf{W})$ but not in $\noc(T',\mathsf{G})$.}

	\medskip
	\noindent
	We now construct one additional such convex set. We distinguish two cases according to the degree of $r'$ in $T'$.

	\smallskip
	\noindent
	\emph{Case (a): $\deg_{T'}(r')=1$.}
	Then $T'$ still has at least three leaves ($r'$ is a special leaf), because $T$ is not a star and some vertex has degree at least~$3$. Consider the set
	\[
	U=V(T')\setminus \{r'\}.
	\]
	This is $P_3$-convex (it induces a subtree) and contains the unique neighbor of $r'$ in $T'$, so it is counted by $\noc(T',\mathsf{W})$ but not by $\noc(T',\mathsf{G})$. Since every $S(v)$ contains at most one leaf, and $U$ contains all leaves of $T'$ except $r'$, we have $U\neq S(v)$ for all $v$. Thus $U$ yields one new convex set.

	\noindent
	\emph{Case (b): $\deg_{T'}(r')\ge 2$.} Since $T$ is not a star, $T'$ has a leaf $z$ at distance at least $2$ from
	$r'$. Let $x$ be a neighbor of $r'$ \textcolor{black}{that does not lie} on the $r'$--$z$ path. Then
	\[
	U=\{z,x\}
	\]
	is \textcolor{black}{a} $P_3$-convex \textcolor{black}{set} (it contains no $P_3$ with both endpoints). It contains exactly one neighbor of the root \textcolor{black}{$r'$}, hence is counted by $\noc(T',\mathsf{W})$ but not by $\noc(T',\mathsf{G})$. Since every \textcolor{black}{set} $S(v)$ induces \textcolor{black}{the unique $v$--$r'$ path, none of them can consist of two vertices in different branches; so, $U\neq S(v)$ for all $v\in V(T')\setminus \{r'\}$.}

	In both cases, we obtain one additional convex set distinct from all the $S(v)$. Therefore,
	\[
	\noc(T',\mathsf{W}) - \noc(T',\mathsf{G})
	\;\ge\;
	(n-2) + 1
	\;=\;
	n-1.
	\]

	\noindent
	\textbf{Part (2).}
	We proceed by induction on $n\ge 1$.

	\paragraph{Base case.} For $n\le 4$ there is no tree that is neither a star nor a path. Thus the statement holds vacuously.

	\paragraph{Inductive step.} Assume $n\ge 5$ and let $T$ be a tree on $n$ vertices which is neither a star nor a path.  Let $r$ be any leaf of $T$, let $r'$ be its unique neighbor, and let $T'=T-r$.  Then $T'$ is a tree of $n-1$ vertices.

	We classify the $P_3$-convex sets of $T$ according to the colors of the edge $rr'$.  The four possible color assignments $(\mathsf{W}/\mathsf{B})$ on $(r,r')$ yield:
	\[
	\begin{array}{c|c}
	(r,r') & \text{contribution} \\ \hline
	(\mathsf{W},\mathsf{W}) & \noc(T',\mathsf{W}) \\
	(\mathsf{W},\mathsf{B}) & \noc(T',\mathsf{B}) \\
	(\mathsf{B},\mathsf{W}) & \noc(T',\mathsf{G}) \\
	(\mathsf{B},\mathsf{B}) & \noc(T',\mathsf{B})
	\end{array}
	\]
	Hence
	\[
	\noc(T)
	=
	\noc(T',\mathsf{W})
	+\noc(T',\mathsf{G})
	+2\noc(T',\mathsf{B}).
	\]
	Since 
	\(
	\noc(T')=\noc(T',\mathsf{W})+\noc(T',\mathsf{B}),
	\)
	we obtain
	\[
	\noc(T)
	=
	2\noc(T') + \bigl( \noc(T',\mathsf{G}) - \noc(T',\mathsf{W}) \bigr).
	\tag{$\ast$}
		\]

	By induction, Lemma \ref{lem:star} and Proposition \ref{pro:path},
	\[
	\noc(T') \;\le\; 2^{n-2} + (n-1),
	\]
	and by the result proved previously,
	\[
	\noc(T',\mathsf{W}) - \noc(T',\mathsf{G}) \;\ge\; n-1,
	\qquad\text{that is,}\qquad
	\noc(T',\mathsf{G}) - \noc(T',\mathsf{W}) \;\le\; -(n-1).
	\]
	Inserting these two bounds into $(\ast)$ yields
	\[
	\noc(T)
	\;\le\;
	2(2^{n-2}+n-1) - (n-1)
	=
	2^{n-1}+n-1.
	\]
	Therefore
	\[
	\noc(T) < 2^{\,n-1} + n = \noc(K_{1,n-1}),
	\]
	completing the proof.
\end{proof}
\textbf{Finally, }the following theorem is an immediate consequence of Lemma~\ref{lem:star}, Corollary~\ref{cor:cycle}, and Propositions~\ref{pro:path},\textcolor{black}{~\ref{prop: technical}} and~\ref{pro:casi}.
\begin{theorem}\label{thm: optimum graph}
	Let $G$ be a connected graph on $n$ vertices. Then
	\[
	\noc(G) \;\le\; 2^{\,n-1} + n.
	\]
	Moreover, equality holds if and only if $G$ is isomorphic to the star $K_{1,n-1}$, the path $P_4$, or the path $P_5$.
\end{theorem}
\section{\textcolor{black}{Complexity results}}\label{sec: Complexity results}
\textcolor{black}{We begin this section by proving that counting $P_3$-convex sets in a graph is $\#\mathsf{P}$-complete. Our reduction relies on the  $\#\mathsf{P}$-completeness of counting independent sets (equivalently, vertex cover)~\cite{Valiant1979}. It is worth noting that this problem remains $\#\mathsf{P}$-complete even when restricted to bipartite graphs~\cite{ProvanB1983}, planar bipartite graphs~\cite{Vadhan2001}, and $k$-regular graphs for any fixed $k\ge 5$~\cite{Vadhan2001}.}

\textcolor{black}{Let $G$ be a graph. Let $\noi(G)$ denote the number of independent sets in $G$.}
\subsection{$\#\mathsf{P}$-completeness}
\begin{theorem}\label{thm: hardness result}
The problem of counting convex sets under $P_{3}$-convexity in split graphs is $\#\mathsf{P}$-complete.
\end{theorem}
\begin{proof}
	We first show that the problem belongs to $\#\mathsf{P}$, and then prove $\#\mathsf{P}$-hardness \textcolor{black}{via} a reduction from counting \textcolor{black}{independent sets}.
	\paragraph{Membership in $\#\mathsf{P}$.}
	\textcolor{black}{An input instance is a graph $H$ (in our reduction, $H$ will be a split graph), and value of the function} is the number of $P_3$-convex vertex sets of $H$. A function $f$ \textcolor{black}{belongs to} $\#\mathsf{P}$ if there exists a polynomial-time decidable relation $R$ such that 
	\[
	\textcolor{black}{f(a) = \bigl|\{b : R(a,b) \text{ holds}\}\bigr|.}
	\]
	We encode a vertex subset $C \subseteq V(H)$ by a bitstring \textcolor{black}{$b$} of length $|V(H)|$, and define \textcolor{black}{$R(H,b)$} to hold if and only if the subset $C$ encoded by \textcolor{black}{$b$} is a $P_3$-convex set of $H$.

	A set $C$ is $P_3$-convex if for every path $P_3 = (x,z,y)$ in $H$ with $x,y \in C$, we also have $z\in C$. This can be checked in $O(n^3)$ time by enumerating all ordered triples $(x,z,y)$, hence \textcolor{black}{$R(H,b)$} is decidable in polynomial time. Therefore, counting $P_3$-convex sets belongs to $\#\mathsf{P}$.
	\paragraph{$\#\mathsf{P}$-hardness via reduction.}
	We reduce from counting independent sets in a graph, which is known to be $\#\mathsf{P}$-complete~\cite{ProvanB1983, Vadhan2001}.
	\paragraph{The case $E=\emptyset$.}
	If $G$ has no edges, then every subset of $V$ is independent and $G$ contains no $P_3$. Thus, every subset of $V$ is also $P_3$-convex. Since an edgeless graph is a split graph (with $K=\emptyset$ and $S=V$), we may simply take $H = G$. In this case,
	\[
	\noc(H) = 2^{|V|} = \noi(G),
	\]
	so the reduction is exact. Hence, for the remainder of the proof, we may assume that $G$ has at least one edge.	

	\subsection*{Reduction for graphs with at least one edge}
	Let $G=(V,E)$ be an input graph with $|E|\ge 1$. We construct a split graph $H$ as follows. The vertex set of $H$ is partitioned into a clique $K$ and an independent set $S$. The clique $K$ contains $|E|+1$ vertices: for each edge $e\in E$ we add a vertex $v_e$, and we add a special universal vertex $u^{*}$. The set $S$ contains $|V|$ vertices, one representing each vertex of $G$. The adjacencies are defined as follows:
	\begin{itemize}
	    \item $K$ is a clique and $u^{*}$ is universal in $H$.
	    \item For every edge \textcolor{black}{$e=xy\in E(G)$}, the vertex $v_e$ is adjacent exactly to the two vertices of $S$ corresponding to $x$ and $y$.
	\end{itemize}
%
	\textcolor{black}{We now count the $P_3$-convex sets of $H$ by partitioning them into three disjoint categories.}
	\paragraph{Case (i): \textcolor{black}{$P_3$-convex sets that contain all vertices of $K$.}}
	Let $S_1=\{s\in S : \deg_H(s)=1\}$. A vertex of $S$ has degree $1$ in $H$ if and only if the corresponding vertex of $G$ has degree $0$. Thus $S_1$ corresponds precisely to 
	\[
	V_0=\{v\in V : \deg_G(v)=0\}.
	\]
	Vertices in $S_1$ may be \textcolor{black}{included or excluded freely} in such a $P_3$-convex set, \textcolor{black}{whereas} all vertices in $S\setminus S_1$ must be included. Hence, there are
	\[
	2^{|S_1|}=2^{|V_0|}.
	\]
	\textcolor{black}{such $P_3$-convex sets.}
	\paragraph{Case (ii): \textcolor{black}{$P_3$-convex sets that contain no vertex of $K$}.}
	\textcolor{black}{Since $u^{*}$ is not included in such a $P_3$-convex set}, \textcolor{black}{it contains} at most one vertex of $S$; otherwise, a $P_3$ would force the inclusion of a vertex of $K$. Hence, \textcolor{black}{there are}
	\[
	|V|+1.
	\]
	\textcolor{black}{such $P_3$-convex sets.}

	\paragraph{Case (iii): \textcolor{black}{$P_3$-convex sets that contain exactly one vertex of $K$.}}
\textcolor{black}{Let $C$ be such a $P_3$ convex set, and let $C\cap K = \{v\}$. In this case we will split the counting into two subcategories.}

	\textcolor{black}{Suppose first that $v\neq u^*$ and $v_e = v$.} If $C$ contained two vertices of $S$, then those two vertices \textcolor{black}{form a $P_3$ with $u^*$, which would force the inclusion of $u^*$ to $C$}, contradicting \textcolor{black}{our} assumption.  
	Thus $|C\cap S|\le 1$, and \textcolor{black}{there are}
	\[
	|E|\cdot|V| + |E|
	\]
	\textcolor{black}{\textcolor{black}{such} $P_3$-convex sets.}

	\textcolor{black}{Suppose now that $v = u^*$}. Then a subset $X \subseteq S$ forms a $P_3$-convex set together with $u^{*}$ if and only if $X$ is an independent set of $G$. Thus, \textcolor{black}{there are $\noi(G)$ such $P_3$-convex sets.}

	Summing \textcolor{black}{up}, we obtain
	\[
	\noc(H) = \noi(G)\;+\;2^{|V_0|}+|V|+1+|E|\cdot|V|+|E|.
	\]

	Therefore, $\noi(G)$ can be computed in polynomial time from \textcolor{black}{$\noc(H)$}, establishing $\#\mathsf{P}$-hardness. Since we also established membership in $\#\mathsf{P}$, the problem is $\#\mathsf{P}$-complete.
\end{proof}
\textcolor{black}{Note that, in the reduction above, the constructed split graph $H$ does not contain two vertex-disjoint induced copies of $K_{1,4}$, since any induced $K_{1,4}$ must have the special universal vertex $u^{*}$ as its center. Hence, we obtain the following strengthened $\#\mathsf{P}$-completeness result.}
\begin{corollary}\label{cor:oneinducedK14}
Even when restricted to split graphs that do not contain two vertex-disjoint induced copies of $K_{1,4}$, 
the problem of counting $P_3$-convex sets is $\#\mathsf{P}$-complete.
\end{corollary}
%
%
%
%
\subsection{Threshold Graphs: a linear-time solvable case}
Throughout this subsection we consider an input threshold graph $G=(V,E)$ under the assumption that $G$ is not isomorphic to the star $K_{1,|V|-1}$.  Stars were analyzed in a dedicated subsection above, and therefore we focus
here solely on non-star threshold graphs. \textcolor{black}{For a thorough treatment of the graph class known as threshold graphs, we refer the reader to~\cite{Golumbikbook}}. 

Threshold graphs form a well-structured subclass of split graphs and admit several equivalent characterizations.  A graph $G$ is a threshold graph if and only if it contains no induced $P_4$, $C_4$, or $2K_2$; equivalently, $G$ can be constructed from a single vertex by repeatedly adding either an isolated vertex or a universal vertex.  Every threshold graph $G$ admits a canonical split partition $(K,S)$ \textcolor{black}{-- a split partition such that $|K| = \omega(G)$ --} where $K$ is a clique, $S$ is an independent set, \textcolor{black}{and $\omega(G)$ denotes the cardinality of a maximum clique}.  
Moreover, the closed neighborhoods of vertices in $S$ are totally ordered
by inclusion, and an analogous nesting relation holds for the neighborhoods
of vertices in $K$ when restricted to~$S$.

If $G$ contains isolated vertices, then each \textcolor{black}{addition} of an isolated vertex doubles the number of $P_3$-convex sets: if $v$ is isolated and $C$ is a convex set of $G-v$, then $C$ and $C\cup\{v\}$ are precisely the convex sets
of $G$ \textcolor{black}{containing the convex set $C$ of $G-v$}. \textcolor{black}{Thus, if 
\[
S_0 = \{\,s\in S : \deg_G(s)=0\,\}.
\]
then 
\[
\noc(G) = 2^{|S_0|}\cdot\noc(G-S_0).
\]
}

Hence, without loss of generality, we assume that the input threshold graph
$G=(V,E)$ has minimum degree $\delta(G)\ge 1$.  
In threshold graphs this implies that $G$ contains exactly $\delta(G)$
universal vertices, all lying in the clique part $K$.  
If $G$ is a complete graph, we take the canonical split partition 
$(K=V,\; S=\emptyset)$.

Let
\[
S_1 = \{\,s\in S : \deg(s)=1\,\}.
\]
Clearly, $S_1=\emptyset$ if and only if $\delta(G)\ge 2$.

\medskip
For the canonical split partition $(K,S)$, every $P_3$-convex set $C$ satisfies
\[
|C\cap K|\in\{0,1,|K|\}.
\]
We analyze these cases separately.

\paragraph{Case (i): $C\cap K=\emptyset$.}
If $C$ contains no vertices of $K$, then $C$ contains at most one vertex of $S$. Indeed, if $s_i,s_j\in C$ with $i\neq j$, then any universal vertex $u\in K$ lies on the path $s_i - u - s_j$, forcing $u\in C$ and contradicting $C\cap K=\emptyset$. Thus this case contributes exactly
\[
|S|+1.
\]
\paragraph{Case (ii): $|C\cap K|=1$.}
Let $C\cap K=\{u\}$.

\textcolor{black}{Let $v$ be a universal vertex of $G$.}

\textcolor{black}{If $v$ is its only universal vertex} and $u=v$, then $C$ may contain more
than one vertex of $S$ but \textcolor{black}{none of them in} $S\setminus S_1$.  
The number of convex sets in this subcase is therefore
%
\textcolor{black}{\[
2^{|S_1|}.
\]}
%
\textcolor{black}{If $u\neq v$} or $G$ has more than one universal vertex, then $C$  \textcolor{black}{contains}  \textcolor{black}{no vertex of $S$}. For such $u$, the number of convex sets is \textcolor{black}{exactly one.}
%

\paragraph{Case (iii): $K\subseteq C$.}
If $C$ contains all vertices of $K$, then every $s\in S$ of degree at least~2
must lie in $C$, since it is the middle vertex of a $P_3$ with both endpoints
in $K$.  
Vertices $s\in S_1$ may be included or excluded independently, as no
$P_3$ forces their presence.  
Thus this case contributes
\[
2^{|S_1|}.
\]

\medskip
Combining all three cases yields the following result.

\begin{theorem}\label{thm: thershold graphs}
Let $G=(V,E)$ be a threshold graph with minimum degree $\delta(G)\ge 1$, not 
isomorphic to a star.  
The
number of $P_3$-convex sets of $G$ is
\textcolor{black}{
\[
\noc(G)=(|S|+1)
\;+\;
|K|
\;+\;
N_U
\;+\;
2^{|S_1|},
\]
}
where
\textcolor{black}{
\[
N_U = 
\begin{cases}
0 & \text{if } \delta(G)\ge 2,\\[2mm]
2^{|S_1|} - 1 & \text{if } \delta(G)=1.
\end{cases}
\]}
The total number can be computed in $O(|V|+|E|)$ time.
\end{theorem}

\begin{proof}
The correctness \textcolor{black}{of the formula follows from the case analysis above.} 
%

Threshold graphs admit linear-time recognition as well as linear-time construction of their canonical split partition.  
All quantities $K$, $S$, $S_1$, and $N_U$ can be computed in $O(|V|+|E|)$ time, and the summation over $w\in K$ also requires linear time. \textcolor{black}{Therefore,} the entire computation runs in $O(|V|+|E|)$ time.
\end{proof}
\section{\textcolor{black}{Exponential-Time Algorithms for Computing $\noc(G)$}}
\label{subsec:exact-algorithms}
In this \textcolor{black}{section} we develop exact exponential-time algorithms for computing $\noc(G)$ for an arbitrary input graph $G=(V,E)$. We begin with a naive $O^*(2^{|V|})$ algorithm and then present more sophisticated algorithms whose time bounds improve the exponential base by combining structural decomposition, forced-color propagation, and fast 
independent-set counting.

Throughout this subsection, we use the two-color interpretation of $P_3$-convexity \textcolor{black}{introduced in Subsection~\ref{subec:trees}.}: 
%
%
A direct enumeration algorithm proceeds as follows. Each vertex independently chooses one of the two colors $\mathsf{B}$ or $\mathsf{W}$, resulting in $2^{|V|}$ possible colorings. For each coloring, we check in $O(|V|+|E|)$ time whether every white vertex has at most one black neighbor. Valid colorings correspond exactly to $P_3$-convex subsets of $G$.
%
\textcolor{black}{
This procedure runs in
\[
O\!\left(2^{|V|}(|V|+|E|)\right)=O^*(2^{|V|}),
\]
time, and will serve as a baseline for comparison.}
\subsection{General strategy for improved exact algorithms}
\label{subsec:general-strategy}
Let $I$ be an independent set of $G$. Assume we enumerate all valid partial colorings $\pi$ of $G-I$.
For each $\pi$ we apply the two propagation rules:
\begin{enumerate}
\item[(R1)] If a white vertex has a black neighbor, then all its 
uncolored neighbors become white.
\item[(R2)] If an uncolored vertex has two black neighbors, it becomes black.
\end{enumerate}

Let $I_\pi$ be the subset of vertices of $I$ that become forced-colored by 
(R1)--(R2).  
If the propagated coloring is invalid, discard~$\pi$.
Otherwise:
%
%
We construct \textcolor{black}{the auxiliary graph} $H_\pi$ as follows:
\begin{itemize}
\item $V(H_\pi)=I\setminus I_\pi$,
\item $xy\in E(H_\pi)$ iff the corresponding vertices of $G$ share a 
white neighbor under $\pi$.
\end{itemize}

Two vertices of $I\setminus I_\pi$ cannot both be black in a valid 
completion exactly when they are adjacent in $H_\pi$.  
Hence every valid completion corresponds to an independent set of $H_\pi$.

Using the $O^*(1.2356^{n})$ algorithm of
Gaspers and Lee~\cite{GaspersLee2023} for a graph with $n$ vertices, we can count the independent sets 
of $H_\pi$ efficiently.

Thus:

\[
\noc(G)=\sum_{\pi}\noi(H_\pi),
\]
where the sum ranges over all valid partial colorings $\pi$ of $G-I$.

\subsection{Running time as a function of a large independent set}\label{subse: general running time for stable sets}
The next exact algorithm follows a general scheme based on a large independent set $I$ of the input graph $G=(V,E)$. Once such a set $I$ is identified, the algorithm proceeds as follows:
\begin{itemize}
    \item enumerate \textcolor{black}{all the valid} partial colorings of $G-I$ into colors $\{\mathsf{B},\mathsf{W}\}$;
    \item for each partial coloring $\pi$, apply the propagation rules (R1)--(R2) to force additional assignments on vertices of $I$;
    \item let $I_\pi\subseteq I$ be the set of vertices of $I$ that became colored by propagation; if the resulting extension is invalid, discard $\pi$;
    \item if $I\setminus I_\pi=\emptyset$, then $\pi$ contributes exactly one $P_3$-convex set;
    \item otherwise, construct the auxiliary graph $H_\pi$ on vertex set $I\setminus I_\pi$, where two vertices are adjacent if and only if their corresponding vertices in $G$ share a white neighbor under~$\pi$; the valid completions of $\pi$ then correspond bijectively to the independent sets of $H_\pi$, which can be counted in time $O^*(1.2356^{\,|I\setminus I_\pi|})$ using the algorithm of Gaspers and Lee~\cite{GaspersLee2023}.
\end{itemize}
\paragraph{A generic upper bound.}
Let $n=|V|$ and $k=|I|$. In the worst case, the number of possible colorings of $G-I$ is $2^{\,|V\setminus I|}=2^{n-k}$; only some of them are valid, but this bound is sufficient to obtain an upper estimate of the running time. For each valid partial coloring $\pi$ on $G-I$, the work needed to process the corresponding $(I\setminus I_\pi)$ is at most $O^*(1.2356^{\,|I\setminus I_\pi|})\le O^*(1.2356^{\,k})$. Therefore, the total running time satisfies
\[
T(G)
= O^*\!\left(2^{\,n-k} \cdot 1.2356^{\,k}\right).
\]
Equivalently, if we set $\alpha = k/n$ (the fraction of vertices in~$I$), we obtain
\[
T(G)
= O^*\!\left(
\bigl(2^{1-\alpha}\cdot 1.2356^{\,\alpha}\bigr)^n
\right).
\]

Thus, the larger the independent set $I$ we can guarantee, the smaller the
exponential base.

\paragraph{Graph classes with guaranteed large independent sets.}
This observation suggests that the algorithm is particularly effective on
graph classes where:

\begin{enumerate}
    \item there exists a constant $\alpha>0$ such that every $n$-vertex graph
          in the class admits an independent set of size at least $\alpha n$, and
    \item such an independent set can be found in polynomial time.
\end{enumerate}

In this situation, we have a uniform bound
\[
T(G)
= O^*\!\left(
\bigl(2^{1-\alpha}\cdot 1.2356^{\,\alpha}\bigr)^n
\right),
\]
for all graphs $G$ in the class.

We briefly illustrate the effect of different values of~$\alpha$:

\begin{itemize}
    \item If $\alpha = 1/2$, then
    \[
    T(G)
    = O^*\!\left(
      \bigl(2^{1/2}\cdot 1.2356^{1/2}\bigr)^n
    \right)
    \approx O^*(1.57^n).
    \]
    \item If $\alpha = 1/3$, then
    \[
    T(G)
    \approx O^*(1.70^n).
    \]
    \item If $\alpha = 1/4$, then
    \[
    T(G)
    \approx O^*(1.77^n).
    \]
\end{itemize}

In many well-studied classes of graphs such an $\alpha$ and a corresponding
polynomial-time construction of a large independent set are known; for example:

\begin{itemize}
    \item Any graph on $n$ vertices and maximum degree $\Delta$, admits an independent set of size at least $n/(\Delta+1)$, which can be obtained by a straightforward greedy algorithm.
		\item Every bipartite graph on $n$ vertices contains an independent set of size at least $n/2$, since one of its two color classes has size at least $n/2$.
    \item For planar graphs, \textcolor{black}{the four colors theorem} guarantees the existence of an independent set of size at least $n/4$~\cite{RobertsonSST1997}, and such a set can be computed efficiently via a known algorithm \textcolor{black}{for four coloring planar graphs~\cite{RobertsonSST1996}}.
\end{itemize}
For each such class, plugging the corresponding lower bound on $|I|$ into the
generic expression
\[
T(G)=O^*\!\left(2^{\,n-|I|}\cdot 1.2356^{\,|I|}\right)
\]
yields an explicit exponential base strictly smaller than~2.
This shows that the general independent-set–based strategy underlying
Algorithm~\ref{alg:second} can be sharpened significantly on classes of graphs
where large independent sets are guaranteed and can be found in polynomial time.
\subsection{Three faster exponential-time algorithms}\label{subsec:second-exact}
We now present three improved algorithms based on a careful decomposition of $G$ 
into three phases until only an independent set $I$ remains.  
For every partial coloring of $G-I$ we invoke the general strategy from 
Subsection~\ref{subsec:general-strategy}. There are three different alternatives for Phase~2, yielding three different time bounds.

\paragraph{Phase 1: Major blocks.}

We search for a vertex $v$ whose $G[N(v)]$ contains a nontrivial 
connected component.  
Among all such vertices we pick one maximizing the size of such a component.
Let $M$ be the induced subgraph consisting of $v$ and that component.  
We remove $M$ and call it a \emph{major block}.  
Let $p$ be the total number of removed vertices.

Every major block satisfies $|M|\ge 3$.  
Every valid coloring of $M$ is exactly one of:

\[
\text{all white},\qquad
\text{exactly one black vertex},\qquad
\text{all black},
\]

\noindent hence exactly $|M|+2$ valid colorings.  
Since $|M|\ge 3$ we have $|M|+2 \le 5^{|M|/3}$, therefore:

\[
\#\text{valid colorings from Phase 1 is at most } 5^{p/3}.
\]

The remainder after Phase~1 is triangle-free (otherwise a triangle would 
produce a bigger block).

\paragraph{Phase 2: three flavors of star extraction.}

In the triangle-free remainder, we consider three variants of Phase~2.

\smallskip
\emph{Variant A (claws $K_{1,3}$).}
Any vertex of maximum degree at least $3$ induces a claw ($K_{1,3}$).  
Whenever such a vertex exists, we remove the entire claw.  
Let $q$ be the total number of removed vertices.  
Each claw has exactly $12$ valid $P_3$-convex colorings, so:
\[
\#\text{valid colorings from Phase 2A is at most } 12^{q/4}.
\]
After this phase the remainder is triangle-free and has maximum degree $\Delta\le 2$, hence it is a disjoint union of paths and cycles.

\smallskip
\emph{Variant B (stars $K_{1,4}$).}
Instead of claws, we iteratively remove induced stars $K_{1,4}$, that is, 
whenever a vertex of degree at least $4$ exists, we pick one such vertex $v$ 
and remove $v$ together with four of its neighbors. 
Again let $q$ be the number of removed vertices.  
Each $K_{1,4}$ has exactly $21$ valid $P_3$-convex colorings, so
\[
\#\text{valid colorings from Phase 2B is at most } 21^{q/5}.
\]
After Phase~2B the remainder is triangle-free with maximum degree $\Delta\le 3$.

\smallskip
\emph{Variant C (stars $K_{1,5}$).}
Analogously, we may remove induced stars $K_{1,5}$, each consisting of a center
and five leaves.  
Let $q$ be the total number of removed vertices.  
Each $K_{1,5}$ has exactly $38$ valid $P_3$-convex colorings, hence
\[
\#\text{valid colorings from Phase 2C is at most } 38^{q/6}.
\]
After Phase~2C the remainder is triangle-free with maximum degree $\Delta\le 4$.

\paragraph{Phase 3: Independent set extraction.}

We find an independent set in the remaining triangle-free graph.  We distinguish two cases depending on the maximum degree $\Delta$ of the remainder.

\begin{itemize}
\item \textcolor{black}{If $\Delta\leq 2$ (Variant A), then the remainder is a disjoint union of paths and cycles.  We extract a maximum independent set from each component: For a path or an even cycle, at least half the vertices lie in $I$;
for an odd cycle $C_{2k+1}$ with $k\ge 2$ (no triangles remain), exactly $k$  vertices lie in $I$.}
%
\item If $3\leq \Delta\leq 4$ (Variants B and C), we apply the standard greedy procedure: iteratively add a minimum-degree vertex $v$ to $I$ and delete $v$ and its neighbors.  This guarantees
\[
|I| \;\ge\; \frac{|V|-p-q}{\Delta+1}.
\]
\end{itemize}

Let $r=|I|$, and let 
\[
t = |V| - p - q - r
\]
be the number of vertices that are removed in Phase~3 and do not belong to $I$. Define $s=t+r$.
Each of these $t$ vertices contributes at most $2$ possible colors, giving:

\[
\#\text{phase-3 colorings is at most } 2^{t}.
\]

\paragraph{Algorithms.}
See Algorithm \ref{alg:second}

\begin{algorithm}[h]
\caption{Improved exact algorithms for computing $\noc(G)$}
\label{alg:second}
\begin{algorithmic}[1]
\Function{Second-Exact-NOC}{$G$}
    \State Perform Phase 1 (major blocks)
    \State Choose one of the three variants of Phase 2 (claws, $K_{1,4}$'s, or $K_{1,5}$'s) and apply it
    \State Perform Phase 3 to obtain an independent set $I$
    \State $count \gets 0$
    \For{each valid coloring of Phase 1 blocks ($\le 5^{p/3}$)}
        \For{each valid coloring of Phase 2 structures
              ($\le 12^{q/4}$ or $\le 21^{q/5}$ or $\le 38^{q/6}$)}
            \For{each coloring of Phase 3 removed vertices ($\le 2^t$)}
                \State form partial coloring $\pi$ on $G-I$
                \State apply rules (R1)--(R2) to propagate forced colors, obtain $\pi'$ and $I_\pi$
                \If{$\pi'$ invalid} \textbf{continue} \EndIf
                \State $J = I \setminus I_\pi$
                \If{$J=\emptyset$}
                    \State $count\gets count+1$
                \Else
                    \State build $H_\pi$ on $J$
                    \State $count \gets count + \noi(H_\pi)$
                \EndIf
            \EndFor
        \EndFor
    \EndFor
    \State \Return $count$
\EndFunction
\end{algorithmic}
\end{algorithm}

\paragraph{Correctness.}

\begin{theorem}
Each of the algorithms obtained by choosing one of the three variants of Phase~2 in Algorithm~\ref{alg:second} computes $\noc(G)$ exactly.
\end{theorem}

\begin{proof}
Every major block and every extracted star (claw $K_{1,3}$, $K_{1,4}$ or $K_{1,5}$) is colored in all valid ways.  
Structural deletions preserve bijections between partial colorings of $G-I$ 
and the valid colorings of the removed structures.

Rules (R1)--(R2) capture all forced consequences of $P_3$-convexity, because:
white vertices cannot have two black neighbors, and vertices with two black 
neighbors must be black.

Vertices of $I$ form an independent set, so the only conflicts among them arise 
when two vertices share a white neighbor, precisely encoded in $H_\pi$.  
Thus every valid completion corresponds to an independent set of $H_\pi$, and 
independent sets of $H_\pi$ correspond to valid completions.  
Every convex set is generated exactly once.
\end{proof}

\paragraph{Running time analysis.}

The total cost of each of the three algorithms arising from Algorithm~\ref{alg:second} is:

\begin{itemize}
\item \textbf{Variant A (claws).}
\[
O^*\!\left( 5^{p/3} \cdot 12^{q/4} \cdot 2^{t} \cdot 1.2356^{\,r} \right),
\qquad p+q+t+r=|V|.
\]

After Phases~1--2A, every component of the subgraph induced by $s=t+r$ vertices is a
path or a cycle.  
Except for odd cycles, at least half of each component's vertices lie in $I$.
For an odd cycle $C_{2k+1}$ with $k\ge 2$ (recall the graph is triangle-free),
exactly $k$ of its vertices lie in $I$.  
The worst case occurs when the entire residual is a disjoint union of $C_5$’s, so:

\[
t=\frac{3s}{5},\qquad r=\frac{2s}{5}.
\]

Then:

\[
\begin{aligned}
5^{p/3}\,12^{q/4}\,2^{t}\,1.2356^{r}
&=
5^{p/3}\,12^{q/4}\,(8\cdot 1.2356^{2})^{s/5}
\\
&\approx
5^{p/3}\,12^{q/4}\,(1.6496)^s.
\end{aligned}
\]

The worst case is when $q=|V|$, which yields

\[
\boxed{O^*(1.86121^{\,|V|})}.
\]

\item \textbf{Variant B ($K_{1,4}$).}
\[
O^*\!\left( 5^{p/3} \cdot 21^{q/5} \cdot 2^{t} \cdot 1.2356^{\,r} \right),
\qquad p+q+t+r=|V|, \; r\geq \frac{(t+r)}{3+1}.
\]

So, the worst case is when:

\[
t=\frac{3s}{4},\qquad r=\frac{s}{4}.
\]

Then:

\[
\begin{aligned}
5^{p/3}\,21^{q/5}\,2^{t}\,1.2356^{r}
&=
5^{p/3}\,21^{q/5}\,(8\cdot 1.2356)^{s/4}
\\
&\approx
5^{p/3}\,21^{q/5}\,(1.77314)^s.
\end{aligned}
\]

The worst case is when $q=|V|$, which yields

\[
\boxed{O^*(1.83842^{\,|V|})}.
\]

\item \textbf{Variant C ($K_{1,5}$).}
\[
O^*\!\left( 5^{p/3} \cdot 38^{q/6} \cdot 2^{t} \cdot 1.2356^{\,r} \right),
\qquad p+q+t+r=|V|, \; r\geq \frac{(t+r)}{4+1}.
\]

So, the worst case is when:

\[
t=\frac{4s}{5},\qquad r=\frac{s}{5}.
\]

Then:

\[
\begin{aligned}
5^{p/3}\,38^{q/6}\,2^{t}\,1.2356^{r}
&=
5^{p/3}\,38^{q/6}\,(16\cdot 1.2356)^{s/5}
\\
&\approx
5^{p/3}\,38^{q/6}\,(1.8164)^s.
\end{aligned}
\]

The worst case is when $q=|V|$, which yields

\[
\boxed{O^*(1.83357^{\,|V|})}.
\]

\end{itemize}
If the input graph contains only $O(1)$ disjoint claws (disjoint $K_{1,4}$'s) (disjoint $K_{1,5}$'s), one has $q=O(1)$  (meaning that the corresponding Phase~2 has polynomial cost),
and the running time improves to

\[
O^*(1.71^{|V|})\;(O^*(1.77314^{|V|}))\;(O^*(1.8164^{|V|})).
\]

\medskip

\noindent\textbf{Comment 1.}
Even split graphs containing only \emph{one} vertex–disjoint induced $K_{1,4}$
already yield $\#\mathsf{P}$-completeness for counting $P_3$-convex sets
(Corollary~\ref{cor:oneinducedK14}). Since the classes of graphs with a bounded
number of vertex–disjoint induced $K_{1,4}$'s (or vertex–disjoint induced
$K_{1,5}$'s) are superclasses of this restricted graph family, the problem
remains $\#\mathsf{P}$-complete for them as well. This provides further
justification for the improved exact algorithms discussed above.
 
\medskip
\noindent
Now, with the goal of reducing the cost of Phase~3, we examine whether it is
possible to guarantee, for every triangle-free graph on $n$ vertices and maximum
degree $\Delta = 3$ (respectively, $\Delta = 4$), an independent set strictly
larger than $n/(\Delta+1)$. If such a guarantee were possible, then the bounds
$O^*(1.77314^{|V|})$ (respectively, $O^*(1.8164^{|V|}))$ could be improved
accordingly.

\medskip
\noindent

\medskip
\noindent
A possible alternative route toward such improvements is the following.
We know that:  
(i) in any connected irregular graph, the greedy algorithm that repeatedly
selects a minimum-degree vertex for inclusion in the independent set $I$ and
then deletes it together with its neighbors guarantees $|I| \ge n/\Delta$;  
(ii) if the graph is bipartite, we automatically obtain
$|I| \ge n/2 \ge n/\Delta$.  
Therefore, the only critical cases are $\Delta$-regular,
\emph{non-bipartite}, triangle-free graphs.

\medskip
\noindent
On the other hand, if we apply the greedy algorithm to a connected
$\Delta$-regular, \emph{non-bipartite}, triangle-free graph
$G=(V,E)$, we obtain an independent set of size at least
\[
\frac{|V|-\Delta-1}{\Delta} + 1
= \frac{|V|-1}{\Delta}
\ge \frac{|V|}{\Delta} - 1,
\]
because after the first iteration, every connected component of the remaining
graph is irregular.  
This means that the gap between $n/\Delta$ and the size of the independent set
produced by the greedy algorithm on a $\Delta$-regular,
\emph{non-bipartite}, triangle-free graph $H$ on $n$ vertices is at most $k$,
where $k$ is the number of connected components of $H$.

\medskip
\noindent
This upper bound can be refined further: the gap is at most the number of
connected components $H_i$ of $H$ that fail to obtain an independent set of
size at least $n_i/\Delta$, where $n_i$ is the number of vertices of $H_i$.
If we can prove that every failing component must have at least $n_\Delta$
vertices, then the greedy algorithm yields an independent set $I$ in $H$ of size at least
\[
\frac{n}{\Delta} - \frac{n}{n_\Delta}
    = \frac{n \,(n_\Delta - \Delta)}{\Delta \,n_\Delta}.
\]

\medskip
\noindent
The idea is therefore to determine the smallest connected 
$\Delta$-regular, \emph{non-bipartite}, triangle-free graph
$G=(V,E)$ that fails to obtain an independent set of size at least
$|V|/\Delta$, so that $n_\Delta = |V|$, or alternatively to infer a good lower
bound on $n_\Delta$ through computational exploration.

\subsubsection{A general scheme for $(k,\ell)$-graphs}
\label{subsec:k-l-scheme}

We now describe how the generic independent-set based strategy from
Subsection~\ref{subsec:general-strategy} can be specialised to
$(k,\ell)$-graphs.  
Recall that a graph $G=(V,E)$ is a $(k,\ell)$-graph if its vertex set
admits a partition
\[
V = A_1 \uplus \cdots \uplus A_k \uplus C_1 \uplus \cdots \uplus C_\ell
\]
such that each $A_i$ is an independent set and each $C_j$ is a clique.
Here $k$ and $\ell$ are fixed constants.

In general, recognising $(k,\ell)$-graphs is NP-complete as soon as $k\ge 3$ or $\ell\ge 3$ (see, e.g.,~\cite{arxiv2510}). In contrast, when $k\le 2$ and $\ell\le 2$ the recognition problem is polynomial-time solvable and, moreover, a corresponding partition can also be constructed in polynomial time\textcolor{black}{~\cite{BrandstadtLeS1998, Brandstadt1998, FederHellKleinMotwani1999, arxiv2510}}. Therefore, in what follows we assume that either
\begin{itemize}
    \item we are given as part of the input a $(k,\ell)$-partition of $V$
          (for hard recognition cases with $k\ge 3$ or $\ell\ge 3$), or
    \item we are working in one of the polynomial-time recognisable subclasses with $k\le 2$ and $\ell\le 2$, in which case a
    suitable partition can be found efficiently.
\end{itemize}

\paragraph{color patterns on the clique side.}
Fix a $(k,\ell)$-partition of $V$.  
Consider one of the clique parts $C_j$.  
By $P_3$-convexity, the behaviour of a convex set on $C_j$ is very restricted:
every vertex of $C_j$ is adjacent to all other vertices of $C_j$, so a white
vertex in $C_j$ can have at most one black neighbour inside $C_j$.
Consequently, the induced coloring of $C_j$ in any $P_3$-convex set is one of:
\begin{itemize}
    \item all vertices of $C_j$ are white;
    \item exactly one vertex of $C_j$ is black;
    \item all vertices of $C_j$ are black.
\end{itemize}
Thus $C_j$ admits at most $|C_j|+2$ possible local color patterns compatible
with $P_3$-convexity.  
Since $k$ and $\ell$ are fixed, the total number of global choices for all
clique parts
\[
(C_1,\dots,C_\ell)
\]
is bounded by
\[
\prod_{j=1}^{\ell} (|C_j|+2),
\]
which is polynomial in $n=|V|$.  
Therefore we may enumerate all color patterns on the clique side in
polynomial time and, for each such pattern, treat it as fixed in what
follows.

\paragraph{Reduction to an independent-set problem on the $A_i$'s.}
Once the colors on all cliques $C_1,\dots,C_\ell$ have been fixed, we are
left with the vertices in
\[
A := A_1 \uplus \cdots \uplus A_k,
\]
which induce the disjoint union of $k$ independent sets.  
On this side the main task is to choose additional black vertices while
respecting $P_3$-convexity.  
In particular, our general framework from
Subsection~\ref{subsec:general-strategy} tells us that the running time
depends crucially on the size of an independent set $I$ that we can
guarantee and construct.

A natural and simple choice is to take $I$ as the largest part among
$A_1,\dots,A_k$.  
Let $n_A = |A|$ be the total number of vertices in the independent side.
Then
\[
|I| \;\ge\; \frac{n_A}{k}.
\]

Applying the generic running-time bound \textcolor{black}{(see Subsection~\ref{subse: general running time for stable sets})}
\[
T(G)
 = O^*\!\left(2^{\,n_A-|I|}\cdot 1.2356^{\,|I|}\right)
\]

Clearly, the worst-case is $|I|=\frac{n_A}{k}$ and $n_A=n$.

\[
T(G)
 = O^*\!\left(2^{\,\frac{n_A(k-1)}{k}}\cdot 1.2356^{\,\frac{n_A}{k}}\right)
 = O^*\!\left(2^{\,\frac{n(k-1)}{k}}\cdot 1.2356^{\,\frac{n}{k}}\right)
 = O^*\!\left(
      \bigl(2^\frac{k-1}{k}\cdot 1.2356^\frac{1}{k}\bigr)^n
   \right).
\]

%
Since the number of color patterns on the cliques is only polynomial,
this bound describes the overall running time on $(k,\ell)$-graphs when
$k$ and $\ell$ are fixed constants.

\paragraph{Range of parameters for which the scheme is useful.}
The base
\[
\beta_k = 2^{(k-1)/k}\cdot 1.2356^{1/k}
\]
is a non-decreasing function of $k$, and for large $k$ the improvement
over the general bound $O^*(1.83357^n)$ disappears. Using the bound above, one checks that the resulting
base is strictly below~$1.83357$ for $k\le 5$.

We emphasise that several important graph classes appear as special
instances of $(k,\ell)$-graphs.  
For example, split graphs are exactly the $(1,1)$-graphs, and bipartite
graphs correspond to $(2,0)$-graphs.
\section*{Acknowledgment}
Mitre C. Dourado is partially supported by Conselho Nacional de Desenvolvimento Cient\'ifico e Tecnol\'ogico (CNPq), Brazil, grant numbers 403601/2023-1 and 305141/2024-4. Luciano Grippo is partially supported by CONICET (PIP 2022-2024 GI - 11220210100392CO). Min Chih Lin is partially supported by CONICET (PIP 11220200100084CO), UBACyT (20020220300079BA and 20020190100126BA), and ANPCyT (PICT-2021-IA-00755). F\'abio Protti is partially supported by Conselho Nacional de Desenvolvimento Científico e 609 Tecnológico (CNPq), Brazil, grant numbers 304643/2025-4 and 409140/2025-2.


\end{document}